\documentclass[12pt,reqno]{amsart}
\usepackage{amsmath}
\usepackage{}
\usepackage{}
\usepackage{amsfonts}
\usepackage{a4wide}
\usepackage{amsmath,amsthm,amssymb,amscd}
\usepackage{latexsym}
\usepackage{hyperref}
\usepackage[numbers,sort&compress]{natbib}
\usepackage{hypernat}
\allowdisplaybreaks
\numberwithin{equation}{section}

\newtheorem{theorem}{Theorem}[section]
\newtheorem{proposition}[theorem]{Proposition}

\newtheorem{lemma}[theorem]{Lemma}

\theoremstyle{definition}

\newtheorem{definition}[theorem]{Definition}

\newtheorem{remark}[theorem]{Remark}

\def\R{\mathbb R}
\def\N{\mathbb N}
\def\C{\mathbb C}

\usepackage{color}

\begin{document}

\title[Logarithmic Schr\"odinger equations]
{Semiclassical states for fractional logarithmic Schr\"{o}dinger equations}

 \author{Xiaoming An}

 \address{School of Mathematics and Statistics\, \&\, Guizhou University of Finance and Economics, Guiyang, 550025, P. R. China}
\email{xman@mail.gufe.edu.cn}

%
%
%

%

\begin{abstract}
In this paper, we consider the following fractional logarithmic Schr\"odinger equation
\begin{equation*}
\varepsilon^{2s}(-\Delta)^s u + V(x)u=u\log |u|^2\ \ \text{in}\ \R^N,
\end{equation*}
where $\varepsilon>0$, $N\ge 1$, $V(x)\in C(\R^N,[-1,+\infty))$. By introducing an interesting penalized function, we show that the problem has a positive solution $u_{\varepsilon}$ concentrating at a local minimum of $V$ as $\varepsilon\to 0$. There is no restriction on decay rates of $V$, especially it can be compactly supported.


\end{abstract}

\maketitle

\section{Introduction}\label{s1}
We study the following fractional Schr\"odinger equation with logarithmic nonlinear term:
\begin{equation}\label{eq1.1}
\varepsilon^{2s}(-\Delta)^s u + V(x)u=u\log |u|^2,\ \ x \in\ \R^N,
\end{equation}
where $\varepsilon>0$, $N\in \N$, $V(x)\in C(\R^N,\R)$ is a continuous potential. This type of problem comes from the study of standing waves $\psi(x,t) = e^{iEt/\varepsilon}u(x)$ of the following fractional nonlinear Schr\"odinger equation:
\begin{equation}\label{Adeq1.2}
i\frac{\partial \psi}{\partial t} = \varepsilon^{2s}(-\Delta)^s \psi + (V(x)+ E)\psi - f(\psi),
\end{equation}
where $f:\C\to \R$ is a function with $f(z) = g(|z|)z$ and $g:\R^+\to \R$ is a  real function(which is $\log|\cdot|^2$ in \eqref{eq1.1}). For power type nonlinearities, the fractional Schr\"odinger equation was introduced by Laskin (\cite{2}, \cite{3}) as an extension of the classical nonlinear Schr\"{o}dinger equations $s = 1$ in which the Brownian motion of the quantum paths is replaced by a L\`{e}vy flight.

Equation \eqref{eq1.1} is a generalization of the classical Nonlinear  Schr\"odinger Equation with logarithmic nonlinearity:
\begin{equation}\label{AAdeq1.4}
-\varepsilon^{2}\Delta u + V(x)u= u\log |u|^2,\ \ x \in\ \R^N
\end{equation}
which admits applications related to quantum mechanics, quantum optics, nuclear physics, transport and diffusion phenomena, open quantum systems, effective quantum gravity, theory of superfluidity and Bose-Einstein condensation(see \cite{K.G.Zlo-2010,Z.Q.-Wang-C.Zhang-ARMA-2019} and the references therein for more details). For this problem, one can check easily that there always exists $u\in H^1(\R^N)$ such that
$$
\int_{\R^N}|u|^2\log|u|^2=-\infty,
$$
which makes the natural functional corresponding to \eqref{AAdeq1.4} is not well defined in $H^1(\R^N)$. To overcome this difficulty, considering the case that $\varepsilon>0$ is fixed, M. Squassin et al. in \cite{M.Squassina-A.Szulkin-2015}, decompose the functional $I$ into the sum of a $C^1$ functional and a convex l.s.c(short for lower semicontinuous hereafter) functional and then use the Mountain Pass Theorem 3.2 in \cite{A.Szulkin-1986} to find a critical point.
W. Shuai in \cite{W.Shuai-JMP-2021} added some growth conditions such as $\liminf\limits_{|x|\to+\infty}V(x)|x|^{2\sigma}>0$(which makes $\int_{\R^N}|u|^2\log|u|^2$ is $C^1$ in $H^1_V(\R^N)=\{u\in H^1(\R^N):V(x)|u|^2\in L^1(\R^N)\}$) and then proved that the problem has a least energy sign-changing solution and a positive ground state. Considering that $\varepsilon>0$ is fixed and $V(x)\equiv\lambda>0$, Wiliam C. Troy in \cite{W.C.Troy-ARMA-2016} developed  a new comparison method to show that the positive solution of \eqref{AAdeq1.4} is unique up to translation when $1\le N\le 9$; Using the result of Serrin-Tang \cite{J.Serrin-M.Tang-IUMJ-2000}, D'Avenia et al. pointed out in \cite{P.D'Avenia-CCM-2014} that the positive solution of \eqref{AAdeq1.4} is also unique up to translation. Following, letting $u_p$ be the unique positive solution of \eqref{AAdeq1.4}, i.e.,
$$
-\Delta u + \lambda u =|u|^{p-2}{u},
$$
where $p>2$, Wang et al. in \cite{Z.Q.-Wang-C.Zhang-ARMA-2019} proved that $u_p$ will converge to the unique solution of \eqref{AAdeq1.4} in the sense of $C^{2,\alpha}(\R^N)$ if $p\to 2$. Considering the case that $\varepsilon\to 0$, by penalized idea, C. Zhang et al. in \cite{C.Zhang-X.Zhang-CV-2020} showed that \eqref{AAdeq1.4} has solutions $u_{\varepsilon}$ concentrating at various types of topological critical points of $V$ as $\varepsilon\to 0$ provided that $\lim\limits_{|x|\to\infty}V(x)|x|^{-2}>-\infty$($V$ can be unbounded below).

To our best knowledge, there are few results about \eqref{eq1.1} in the nonlocal case $0<s<1$. When $\varepsilon>0$ is fixed, D'Avenia et al. in \cite{P.D'Avenia-MMAS-2015} obtained existence of infinitely weak solutions. In \cite{A.H.Ardila-NA-2017}, it was proved by compactness method that \eqref{eq1.1} has ground states which are stable.

In this paper, we are interesting in semiclassical analysis of \eqref{eq1.1}. From a mathematical point of view, the transition from quantum to classical mechanics can be formally performed by letting $\varepsilon\to 0$. For small $\varepsilon > 0$, solutions $u_{\varepsilon}$ are usually referred to as semiclassical bound states.

In order to state our main result, we need to give some notations and assumptions.
For  $s\in(0,1)$, the fractional Sobolev space $H^s(\mathbb{R}^N)$ is defined as
$$
H^s(\mathbb{R}^N) = \Big\{u\in L^2(\mathbb{R}^N):\frac{u(x) - u(y)}{|x - y|^{N/2 + s}}\in L^2(\mathbb{R}^N\times\mathbb{R}^N)\Big\},
$$
endowed with the norm
$$
\|u\|_{H^s(\mathbb{R}^N)} = \Big(\int_{\mathbb{R}^N}|(-\Delta)^{s/2}u|^2 + u^2\,  {\mathrm{d}}x \Big)^{\frac{1}{2}},
$$
where
$$
\int_{\mathbb{R}^N}|(-\Delta)^{s/2}u|^2{\mathrm{d}}x  = \int_{\mathbb{R}^{2N}}\frac{|u(x) - u(y)|^2}{|x - y|^{N + 2s}}\,{\mathrm{d}}x \,  {\mathrm{d}}y.
$$
Like the classical case, we define the space $\dot{H}^s(\R^N)$ as the completion of $C^{\infty}_c(\R^N)$ under the norm
$$
\|u\|^2 = \int_{\mathbb{R}^N}|(-\Delta)^{s/2}u|^2{\mathrm{d}}x  = \int_{\mathbb{R}^{2N}}\frac{|u(x) - u(y)|^2}{|x - y|^{N + 2s}}\,{\mathrm{d}}x \,  {\mathrm{d}}y.
$$
We will use the following local fractional Sobolev space
$$
W^{s,2}(\Omega) = \Big\{u\in L^2(\Omega): \frac{|u(x) - u(y)|}{|x - y|^{\frac{N}{2} + s}}\in L^2(\Omega\times\Omega)\Big\}.
$$
It is easy to check that $W^{s,2}(\Omega)$ is a Hilbert space under the following inner product
$$
(u,v) = \int_{\Omega}\int_{\Omega}\frac{(u(x) - u(y))(v(x) - v(y))}{|x - y|^{N + 2s}}dxdy + \int_{\Omega}uvdx\ \ \forall u,v\in W^{s,2}(\Omega),
$$
see \cite{4} for more details.
Also from \cite{4},  the fractional Laplacian is defined as
\begin{align*}
(-\Delta)^s u(x) \, &= \, C  (N,s)P.V.\int_{\mathbb{R}^N}\frac{u(x) - u(y)}{|x - y|^{N + 2s}}\,{\mathrm{d}}y
\\[2mm]
\, & =\,   C(N,s)\lim_{\varepsilon\to 0}\int_{\mathbb{R}^N          \backslash        B_{\varepsilon}(x)      }\frac{u(x) - u(y)}{|x - y|^{N + 2s}}\,{\mathrm{d}}y.
\end{align*}
For the sake of simplicity, we define for every $u\in \dot{H}^s(\R^N)$ the fractional $(-\Delta)^s$ as
$$
(-\Delta)^s u(x) = 2\int_{\mathbb{R}^N}\frac{u(x) - u(y)}{|x - y|^{N + 2s}}\,{\mathrm{d}}y.
$$
Note that for every $\varphi\in C^{\infty}_c(\R^N)$, it holds
$$
\int_{\R^N}(-\Delta)^s u(x)\varphi(x) = \frac{d}{dt}\int_{\R^N}|(-\Delta)^{s/2}(u+t\varphi)|^2\large|_{t=0}.
$$
Our solutions will be found in the following weighted fractional Sobolev  space:
$$
\mathcal{D}^s_{V,\varepsilon}(\mathbb{R}^N) = \left\{u\in \dot{H}^s(\R^N):\,\,u\in L^2\big(\mathbb{R}^N,\big(V(x)+1\big)\,{\mathrm{d}}x\big)\right\},
$$
endowed with the norm
$$
\|u\|_{\mathcal{D}^s_{V,\varepsilon}(\mathbb{R}^N)} = \Big(\int_{\mathbb{R}^N}\varepsilon^{2s}|(-\Delta)^{s/2}u|^2 + (V(x) + 1)u^2\,{\mathrm{d}}x \Big)^{\frac{1}{2}}.
$$
For the potential term $V$, we assume that $V$ is continuous and

$(\mathcal{V}_1)$ $V(x)+1\ge 0$;

$(\mathcal{V}_2)$  there exist open bounded sets $\Lambda\subset\subset U$ with smooth boundaries $\partial\Lambda,\ \partial U$, such that
\begin{equation}\label{{Aeq1.3}}
0 < \lambda = \inf_{\Lambda}(V+1) < \inf_{U\backslash\Lambda}(V+1).
\end{equation}
Without loss of generality, we assume that $0\in\Lambda$.

Now, with the notations and assumptions above at hand, we are in a position to state our main result:

\begin{theorem}\label{th1.1}
Let $V$ satisfy $(\mathcal{V}_1)$ and $(\mathcal{V}_2)$. Then there exists an $\varepsilon_0>0$ such that \eqref{eq1.1} has a  positive solution $u_{\varepsilon}$ if $\varepsilon\in(0,\varepsilon_0)$. Moreover, $u_{\varepsilon}$ has a global maximum point $x_{\varepsilon}$ which satisfies $$\lim_{\varepsilon\to 0}V(x_{\varepsilon}) = \inf_{x\in\Lambda}V(x)$$ and
\begin{equation*}\label{eq6.4}
u_{\varepsilon}(x)\le\frac{C\varepsilon^{N+2s}}{\varepsilon^{N+2s}+|x - x_{\varepsilon}|^{N+2s}},\  x\in\R^N,
\end{equation*}
where $C$ is a positive constant.
\end{theorem}

The difficulties in the proof of Theorem \ref{th1.1} are stated as follows.
Firstly, there exists $u\in\mathcal{D}^s_{V,\varepsilon}(\R^N)$ such that
$$
\int_{\R^N}|u|^2\log|u|^2 = -\infty,
$$
which makes the natural Euler-Lagrange functional corresponding to \eqref{eq1.1}, i.e.,
\begin{equation*}
I_{\varepsilon}(u):=\frac{1}{2}\int_{\R^N}\big(\varepsilon^{2s}|(-\Delta)^{s/2}u|^2 + (V(x)+1)|u|^2\big)-\frac{1}{2}\int_{\R^N}u^2\log u^2 \mathrm{d}x,
\end{equation*}
is not well defined in $\mathcal{D}^s_{V,\varepsilon}(\R^N)$. Then all the methods developed for \eqref{eq1.1} in the case that the logarithmic term replaced by power type nonlinear term failed(see \cite{100,APX}). Moreover, the expectation that the concentration should occur at a local minimum of $V$ in $\Lambda$ makes us have to truncate the nonlinear term outside $\Lambda$.
Noting that since the ground states of limiting equation of \eqref{eq1.1} decay like $\frac{1}{|x|^{N+2s}}$, we can not use the penalized idea in \cite{C.Zhang-X.Zhang-CV-2020} which deal with the case $s=1$ to truncate the nonlinear term or the potential term.


Creatively, to overcome the two difficulties above, we use the characteristic function $4\chi_{\R^N\Lambda}$ to truncate the nonlinear term(see \eqref{AAAeq1} below), i.e., we firstly study the following penalized problem
\begin{equation}\label{Revisioneq2.1}
\varepsilon^{2s}(-\Delta)^su_{\varepsilon} + (V(x) + 1)u_{\varepsilon} = \chi_{\Lambda}(1 + \log|u|^2)u_{\varepsilon} - \chi_{\R^N\backslash\Lambda}\max\{2u,-u(1 + \log |u|^2)\}.
\end{equation}
The corresponding functional $J_{\varepsilon}$, all through  is  not well defined in $\mathcal{D}^s_{V,\varepsilon}$, is the sum of a $C^1$ functional and a convex l.s.c. functional(see \eqref{AAAeq1} and Remarks \ref{re2.3} below), whose critical points can be found  by the Mountain Pass Theorem 3.2 in \cite{A.Szulkin-1986}.
Then,  a penalized solution  $u$ of \eqref{Revisioneq2.1} is indeed a solution of the origin problem \eqref{eq1.1} if
$$
-(1+\log|u|^2)\ge 2\ \ \text{on}\ \R^N\backslash\Lambda.
$$
This,  after the proof of concentration of $u$, can be checked easily  by the well-known fact that a sub-solution of $(-\Delta)^su + u\le0$ decays like $\frac{1}{|x|^{N+2s}}$(see the last part of Section 3 for more details).

We need to emphasize that the nonlocal term $(-\Delta)^s$ makes the estimates more difficult than the classical case. In fact, for a smooth function $f\in C^{\infty}_c(\R^N)$, one can not compute $(-\Delta)^sf$ as precisely as $-\Delta f$(see Appendix A below for example).

The nonlocal effect makes us have to know the global $L^2$-norm information of penalized solution, which in \cite{APX} was given by the assumption $\liminf_{|x|\to \infty}(V(x)+1)|x|^{2s}>0$. But, in the present paper, the global $L^2$-norm information of penalized solution will be obtained by the fractional Hardy inequality: There exists a positive constant $C_{N,s}$ such that
\begin{equation}\label{Adeq1.4}
\int_{\R^N}\frac{|u(x)|^2}{|x|^{2s}}\,{\mathrm{d}}x\le C_{N,s}|(-\Delta)^{s/2}|^2_2
\end{equation}
for all $u\in \dot{H}^s(\R^N)$(see the proof of \eqref{Adeq3.3} in Appendix A for example).

\vspace{0.5cm}
\textbf{Plan of the paper.} In Section \ref{s2}, we obtain the penalized problem by truncating the logarithmic term in \eqref{eq1.1} outside by $4\chi_{\R^N\Lambda}u$, then we use the Mountain Pass Theorem 3.2 in \cite{A.Szulkin-1986} to obtain a penalized solution $u_{\varepsilon}$. In Section \ref{s3}, we study the concentration of $u_{\varepsilon}$ and then linearize the penalized equation in Section \ref{s2}. At the last part of Section \ref{s3},  we use the well-known decay estimates of fractional Schr\"oinger equations to prove  the asymptotic behaviour of $u_{\varepsilon}$, which implies that  $u_{\varepsilon}$ solves the origin problem.

\vspace{1cm}

\section{The penalized problem}\label{s2}
The following inequality exposes the relationship between $H^s(\R^N)$ and the Banach space $L^q(\R^N)$.
\begin{proposition}\label{wpr2.1} ({Fractional version of the Gagliardo$-$Nirenberg inequality \cite{4}})
For every $u\in H^s(\mathbb{R}^N)$,
\begin{align} \label{inequality2.1}
\|u\|_{q} \leq C \|(-\Delta)^{s/2}u\|^{\theta}_{2}\|u\|^{1 -\theta }_{2},
\end{align}
where $q\in [2,2^*_s]$ and $\theta $ satisfies $\frac{\theta}{2^*_s} + \frac{(1 -\theta)}{2} = \frac{1}{q}$.
\end{proposition}

By the proposition above, we know that $ H^s(\mathbb{R}^N)$ is continuously  embedded into $L^q(\mathbb{R}^N)$ for $\ q\in [2,2^*_s]$. Moreover, on bounded set, the embedding is compact (see \cite{4}), i.e.,
$$
H^s(\mathbb{R}^N)\subset\subset L^q_{loc}(\mathbb{R}^N)\ \text{compactly, if}\ q\in [1,2^*_s).
$$

Our proof will rely on the following fractional logarithmic Sobolev inequality(see \cite{A.Cotsiilis-N.Tavoularis-2005}):

\begin{proposition}\label{Adpr2.2} For any $u\in H^s(\R^N)$, it holds
$$
\int_{\R^N}|u|^2\log \Big(\frac{|u|^2}{\|u\|^2_2}\Big) + \Big(N+\frac{N}{s}\log a + \log \frac{s\Gamma(\frac{N}{2})}{\Gamma(\frac{N}{2s})}\Big)\|u\|^2_2\le \frac{a^2}{\pi^s}\|(-\Delta)^s\|^2_2,\ a>0.
$$
\end{proposition}

Now we are going to modify the origin problem \eqref{eq1.1}. In order to use the Mountain Pass Theorem 3.2 in \cite{A.Szulkin-1986}, considering the vanishing of $V$ and the concentration should occur in $\Lambda$, we modified the nonlinear term as follows. Define
\begin{align}\label{AAAeq1}
\nonumber G_1(x,s)&=\frac{1}{2}\chi_{\Lambda}(x)s^2_+\log s^2_+,\\  G_2(x,s)&=\frac{1}{2}\chi_{\R^N\backslash\Lambda}(x)\int_{0}^{s}\max\{4t_+,-2t_+(1 + \log t^2_+)\}dt.
\end{align}
\begin{remark}\label{re2.3}
\textbf{I}mportantly, since  $G'_2(x,s)$ is nondecreasing on $s$, $G_2(x,s)\ge 0$ is convex on $s$ for all $x\in\R^N\backslash\Lambda$ and the functional $\mathcal{G}^2:\mathcal{D}^s_{V,\varepsilon}(\R^N)\to \R$ given by
$$
\mathcal{G}^2(u)= \int_{\R^N}G_2(x,u)
$$
is convex and l.s.c by Fatou's Lemma.

Proposition \ref{wpr2.1} and the boundedness of $\Lambda$ imply that the functional $\mathcal{G}^1:\mathcal{D}^s_{V,\varepsilon}(\R^N)\to \R$ given by
$$
\mathcal{G}^1(u)= \int_{\R^N}G_1(x,u)
$$
is $C^1$. So the functional $$J_{\varepsilon}(u)=\Phi_{\varepsilon}(u) + \Psi(u)$$ with $$\Phi_{\varepsilon}(u) = \frac{1}{2}\|u\|^2_{V,\varepsilon}-\mathcal{G}^1(u)\ \ \text{and}\ \ \Psi=\mathcal{G}^2(u)$$
 has the form stated in \cite{A.Szulkin-1986}.
\end{remark}

 By the remark above, although $J_{\varepsilon}$ is not $C^1$, we can still use the Mountain Pass Theorem 3.2 in \cite{A.Szulkin-1986} to find a critical point for $J_{\varepsilon}$. We first state some necessary definitions corresponding to those functionals has the form of $J_{\varepsilon}$.
\begin{definition}\label{de2.1}
Let $E$ be a Banach space, $E'$ be the dual space of $E$ and $\langle \cdot,\cdot\rangle$ be the duality paring between $E'$ and $E$. Let $J:E\to \R$ be a functional of the form $J(u)=\Phi(u)+\Psi(u)$, where $\Phi\in C^1(E,\R)$ and  $\Psi$ is convex and l.s.c.. We have the following definitions:

$(i)$ A critical point of $J$ is a point $u\in E$ such that $J(u)<+\infty$ and $0\in\partial J(u)$, i.e.
$$
\langle\Phi'(u),v-u\rangle + \Psi(v)-\Psi(u)\ge 0,\ \forall v\in E.
$$

$(ii)$ A Palais-Smale sequence at level $c$ for $J$ is a sequence $(u_n)\subset E$ such that $J(u_n)\to c$ and there is a numerical sequence $\sigma_n\to 0^+$ with
$$
\langle\Phi'(u_n),v-u_n\rangle + \Psi(v)-\Psi(u_n)\ge -\sigma_n\|v-u_n\|,\ \forall v\in E.
$$

$(iii)$ The functional $J$ satisfies the Palais-Smale condition at level $c$ $(PS)_c$ condition if all Palais-Smale sequence at level $c$ has a convergent subsequence.

$(iv)$ The set $D(J):=\{u\in E:J(u)<+\infty\}$ is called the effective domain of $J$.

\end{definition}

According to the Corollary 2.6 in \cite{M.Squassina-A.Szulkin-2015}, we have
\begin{proposition}\label{Adpr2.4}
Let $\sup_nJ_{\varepsilon}(u_n)<+\infty$. Then $(u_n)$ is a Palais-Smale sequence if and only if $J'(u_n)\to 0$ in $\big(\mathcal{D}^s_{V,\varepsilon}(\R^N)\big)'$.
\end{proposition}

To use Theorem 3.2 in \cite{A.Szulkin-1986}, we need to prove that $J_{\varepsilon}$ satisfies the $(PS)_c$ condition $(iii)$ above.
\begin{proposition}\label{pr2.2} $J_{\varepsilon}$ satisfies $(PS)_c$ condition, i.e., each sequence $(u_n)\subset \mathcal{D}^s_{V,\varepsilon}(\R^N)$ with $\lim\limits_{n\to\infty}J_{\varepsilon}(u_n)\to c$
has a convergent subsequence in $\mathcal{D}^s_{V,\varepsilon}(\R^N)$.
\end{proposition}
\begin{proof}
We first show that $(u_n)$ is bounded in $\mathcal{D}^s_{V,\varepsilon}(\R^N)$. Observing firstly that $\Psi\ge 0\ \forall u\in\mathcal{D}^s_{V,\varepsilon}(\R^N)$, hence we have
\begin{equation}\label{SSEQ2.2}
J_{\varepsilon}(u_n)\ge \Phi_{\varepsilon}(u_n) = \frac{1}{2}\|u_n\|^2_{V,\varepsilon} - \frac{1}{2}\int_{\Lambda}|u_n|^2\log|u_n|^2.
\end{equation}
Moreover, by \eqref{AAAeq1}, Proposition \ref{Adpr2.4} and the fact that $\int\max\{f,g\}\ge \max\{\int f,\int g\}$, it holds
\begin{equation}\label{AAdeq2.4}
\int_{\R^N}|u_n|^2\mathrm{d}x\le 2J_{\varepsilon}(u_n) - \langle J'_{\varepsilon}(u_n),u_n\rangle\le C + o_n(1)\|u_n\|_{V,\varepsilon}.
\end{equation}

For a set $A\in\R^N$, we define $A^{d}:=\{x\in\R^N:dist(A,x)<d\}$.
Let $\eta\in C^{\infty}_c(\Lambda^{\delta})$ be a function satisfying $0\le \eta\le 1$ and $\eta\equiv 1$ on $\overline{\Lambda}$, where the $\delta$ is a small parameter such that $\Lambda^{2\delta}\subset\subset U$. Defining $v_n(x) = u_n(\varepsilon x)\eta(\varepsilon x)$, we have $v_n(x)\in H^s(\R^N)\backslash\{0\}$. Then, by Proposition \ref{Adpr2.2}, we have
\begin{align*}
\int_{\R^N}|v_n|^2\log|v_n|^2 \le \frac{a^s}{\pi}\|(-\Delta)^{s/2}v_n\|^2 + (C_{s,N} + \frac{N}{s}\log a)\|v_n\|^2_2 + \|v_n\|^2_2\log \|v_n\|^2_2.
\end{align*}
Resclaing back, we find

\begin{eqnarray}\label{Adeq2.4}
\begin{split}
&\quad\int_{\R^N}|\eta u_n|^2\log\|\eta u_n\|^2\\
&\le \frac{a\varepsilon^{2s}}{2}\|(-\Delta)^{s/2}(\eta u_n)\|^2_2 + (C_{s,N} + \frac{N}{s}\log a)\|\eta u_n\|_2^2+ \|\eta u_n\|^2_2 \log \Big(\frac{1}{\varepsilon^N}\|\eta u_n\|^2_2\Big).
\end{split}
\end{eqnarray}
By fractional Hardy inequality \eqref{Adeq1.4} and some delicate nonlocal estimates, we will prove in Appendix A that
\begin{equation}\label{Adeq2.5}
T_1(\eta):=\varepsilon^{2s}\|(-\Delta)^{s/2}(\eta u_n)\|^2_2\le C\|u_n\|^2_{V,\varepsilon},
\end{equation}
where $C$ is a positive constant. Hence, returning back to \eqref{Adeq2.4}, letting $a>0$ be small enough, by \eqref{AAdeq2.4}, we get
\begin{eqnarray}\label{Adeq2.6}
\begin{split}
&\quad\int_{\R^N}|\eta u_n|^2\log\|\eta u_n\|^2\le \frac{1}{2}\|u_n\|^2_{V,\varepsilon} + C(1 + \|\eta u_n\|_2^{2(1 + \delta)}) -N\|u_n\|^2_2\log\varepsilon\\
&\le \frac{1}{2}\|u_n\|^2_{V,\varepsilon} + C(1 + \|u_n\|^{1 + \delta}_{V,\varepsilon}) - N\big(C + \|u_n\|_{V,\varepsilon}\big)\log\varepsilon,
\end{split}
\end{eqnarray}
where $\delta\in(0,1)$ is a parameter.
Note that for every $f:\R^N\to \R$, it holds
$$
\int_{U\backslash\Lambda}|f|^2\log |f|^2\ge -\frac{1}{e}|U\backslash\Lambda|.
$$
Then by \eqref{Adeq2.6} and \eqref{SSEQ2.2}, we have
$$
\widetilde{C}\ge \frac{1}{2}\|u_n\|^2_{V,\varepsilon} - \|u_n\|^{1 + \delta}_{V,\varepsilon} + N(C + \|u_n\|_{V,\varepsilon})\log\varepsilon,
$$
where $\widetilde{C}$ is a positive constant. Consequently, since $0<\delta<1$, we can conclude that $(u_n)$ is bounded in $\mathcal{D}^s_{V,\varepsilon}(\R^N)$.
 Going if necessary to a subsequence, we assume that
 $$
 u_n\rightharpoonup u\ \text{weakly in}\ \mathcal{D}^s_{V,\varepsilon}(\R^N).
 $$

 Next we show that $(u_n)$ has a convergent subsequence. Since $u_n\rightharpoonup u\in \mathcal{D}^s_{V,\varepsilon}(\R^N)$, by the boundedness of $\Lambda$ and the fact that $|G_1(x,s)|\le C_1|t|^{\frac{3}{2}} + C_2|t|^{\frac{5}{2}}$, where $C_1,C_2>0$ are two positive constants, we have
 $$
 {G}_1'(x,u_n)u_n\to {G}_1'(x,u)u.
 $$
Noting that since $J'_{\varepsilon}(u_n)\varphi = o_n(1)\|\varphi\|_{V,\varepsilon}$ for all $\varphi\in C^{\infty}_c(\R^N)$, we deduce that $J'_{\varepsilon}(u)\varphi = 0$ for all $\varphi\in C^{\infty}_c(\R^N)$, and so, $J'_{\varepsilon}(u)u=0$. Combing with $|J'_{\varepsilon}(u_n)u_n|\le \|J'_{\varepsilon}(u_n)\|_{\mathcal{D}'_{V,\varepsilon}}\|u_n\|_{V,\varepsilon} =  o_n(1)\|u_n\|$, we have
\begin{equation}\label{Adeq2.8}
\|u_n\|^2_{V,\varepsilon} + \int_{\R^N}\chi_{\R^N\backslash\Lambda}G_2'(x,u_n)u_n = \|u\|^2_{V,\varepsilon} + \int_{\R^N}\chi_{\R^N\backslash\Lambda}G_2'(x,u)u + o_n(1).
\end{equation}
Finally, by $\|u\|\le \liminf\limits_{n\to\infty}\|u_n\|^2_{V,\varepsilon}$ and the l.s.c property(by Fatou's Lemma), we get that $u_n\to u$ strongly in $\mathcal{D}^s_{V,\varepsilon}(\R^N)$. This completes the proof.
\end{proof}

\begin{remark}
The proof of \eqref{Adeq2.5} is trivial if $s=1$, but delicate if $0<s<1$. It will need the global $L^2$ information of $u_n$, which is given by the nonlocal operator $(-\Delta)^s$ and the fractional Hardy inequality \eqref{Adeq1.4}(see \eqref{eqA.2} for example).
\end{remark}
Obviously, there exists $\rho>0$ such that
$$
J_{\varepsilon}(u)\ge \Phi_{\varepsilon}(u)\ge \frac{1}{2}\|u\|^2_{V,\varepsilon} - C\|u\|^p_{V,\varepsilon}>0\ \text{for all}\ u\ \text{with}\ \|u\|^2_{V,\varepsilon} = \rho
$$
and for each $u\in C^{\infty}_c(\Lambda)\backslash\{0\}$, it holds
$$
J_{\varepsilon}(su)\to-\infty\ \text{as}\ s\to+\infty,
$$
i.e., $J_{\varepsilon}$ owns mountain pass geometry. Thus by Proposition \ref{pr2.2} and Theorem 3.2 in \cite{A.Szulkin-1986}, we immediately have:
\begin{lemma}\label{le2.3}
The mountain pass value
$$
c_{\varepsilon}=\inf_{\gamma\in\Gamma}\max_{t\in[0,1]}J_{\varepsilon}(\gamma(t))
$$
is positive and can be achieved by a positive function $u_{\varepsilon}$ which is a critical point of $J_{\varepsilon}$ and solves the following penalized problem
\begin{equation}\label{eq2.1}
\varepsilon^{2s}(-\Delta)^s u_{\varepsilon} + (V(x) + 1)u_{\varepsilon} = G_1'(x,(u_{\varepsilon})_+) - G_2'(x,(u_{\varepsilon})_+).
\end{equation}
\end{lemma}

\begin{proof}
By Theorem 3.2 in \cite{A.Szulkin-1986}, $c_{\varepsilon}$ is a critical value, i.e., there exists $u_{\varepsilon}\in D(J_{\varepsilon})=\{v\in \mathcal{D}^s_{V,\varepsilon}:J_{\varepsilon}(v)<+\infty\}$ with $J_{\varepsilon}(u_{\varepsilon}) = c_{\varepsilon}$ such that
$$
\langle \Phi'_{\varepsilon}(u_{\varepsilon}),v-u_{\varepsilon}\rangle + \Psi(v) - \Psi(u_{\varepsilon})\ge 0\ \ \forall v\in \mathcal{D}^s_{V,\varepsilon}.
$$
In particular, letting $t>0$ and $v = u_{\varepsilon} + t\varphi$ with $\varphi\in C^{\infty}_c(\R^N)$, we have
$$
\langle \Phi'_{\varepsilon}(u_{\varepsilon}),\varphi\rangle + \frac{\Psi(u_{\varepsilon} + t\varphi) - \Psi(u_{\varepsilon})}{t}\ge 0\ \ \forall t>0
$$
and then
$$
\langle J'_{\varepsilon}(u_{\varepsilon}),\varphi\rangle = \langle \Phi'_{\varepsilon}(u_{\varepsilon}),\varphi\rangle + \int_{\R^N}G_2'(x,u_{\varepsilon})\varphi\ge 0.
$$
Rearranging $\varphi = -\psi$, we eventually have
$$
\langle J'_{\varepsilon}(u_{\varepsilon}),\varphi\rangle = 0\ \forall \varphi\in C^{\infty}_c(\R^N),
$$
which implies \eqref{eq2.1}.

Finally, letting $(u_{\varepsilon})_-$ be a test function to \eqref{eq2.1}, we find $u_{\varepsilon}\ge 0$. By the standard regularity assertion and maximum principle in \cite[Appendix D]{20}, we conclude that $u_{\varepsilon}$ is positive.
\end{proof}

\vspace{0.5cm}

\section{Concentration and the origin problem}\label{s3}

In this section we will prove the concentration phenomenon of $u_{\varepsilon}$ via energy comparison, by which we will prove at the last of this section that
\begin{equation}\label{Adeq3.1}
-(1 + \log|u_{\varepsilon}|^2)\ge 2\ \ \forall x\in\R^N\backslash\Lambda,
\end{equation}
which and \eqref{eq2.1} indicate that $u_{\varepsilon}$ solves the origin problem \eqref{eq1.1}.

\subsection{Concentration}
In the first subsection, we prove the concentration of $u_{\varepsilon}$ via comparing the energy $c_{\varepsilon}$ with the least energy of the limiting problem of \eqref{eq1.1}.

The limiting problem corresponding to \eqref{eq1.1} is
\begin{equation}\label{eq3.1}
(-\Delta)^s u + \lambda u = u\log |u|^2,
\end{equation}
where $\lambda>-1$. Its Euler-Lagrange functional is
$$
\mathcal{L}_{\lambda}(u) = \frac{1}{2}\int_{\R^N}|(-\Delta)^{s/2}u|^2 + (\lambda + 1)|u|^2 - \frac{1}{2}\int_{\R^N}u^2\log u^2.
$$
In \cite{A.H.Ardila-NA-2017}, it was proved that the limiting problem \eqref{eq3.1} has a least energy solution $U_{\lambda}$ with
$$
\mathcal{L}_{\lambda}(U_{\lambda}) = \mathcal{C}_{\lambda} := \inf_{\varphi\in H^s(\R^N)\backslash\{0\}}\max_{t>0}\mathcal{L}_{\lambda}(t\varphi)=\inf_{\varphi\in C^{\infty}_c(\R^N)\backslash\{0\},\varphi\ge0}\max_{t>0}\mathcal{L}_{\lambda}(t\varphi).
$$
For $\mathcal{C}_{\lambda}$, we have
\begin{proposition}\label{pr3.1}
the function $\mathcal{C}_{\cdot}:(-1,+\infty)\to (0,+\infty)$ is continuous and increasing.
\end{proposition}
\begin{proof}
Let $-1<\lambda<\lambda'<+\infty$ and $U_{\lambda'}$ be the ground solution of equation \eqref{eq3.1} with $\lambda=\lambda'$. An easy analysis shows that the following function $f:\R^+\to \R^+$
$$
f(t)=\mathcal{L}_{\lambda}(tU_{\lambda'}) = \mathcal{L}_{\lambda'}(tU_{\lambda'}) + \frac{t^2}{2}(\lambda - \lambda')\int_{\R^N}|U_{\lambda'}|^2
$$
has a unique maximum point $t'\in(0,+\infty)$, from which we have
\begin{align*}
\mathcal{C}_{\lambda}\le \max_{t>0}f(t) &= \mathcal{L}_{\lambda'}(t'U_{\lambda'}) + \frac{(t')^2}{2}(\lambda - \lambda')\int_{\R^N}|U_{\lambda'}|^2\\
&\le \mathcal{C}_{\lambda'} + \frac{(t')^2}{2}(\lambda - \lambda')\int_{\R^N}|U_{\lambda'}|^2\\
&<\mathcal{C}_{\lambda'}.
\end{align*}
Similarly, it holds
$$
\mathcal{C}_{\lambda'}\le \mathcal{C_{\lambda}} + \frac{\tilde{t}^2}{2}(\lambda' - \lambda)\int_{\R^N}|U_{\lambda}|^2
$$
for some unique $\tilde{t}\in(0,+\infty)$.
Then  $\mathcal{C}_{\lambda}$ is increasing and continuous.
\end{proof}

By the analysis above, we have the following upper bound of $c_{\varepsilon}$.
\begin{proposition}\label{pr3.2}
It holds
$$
\limsup_{\varepsilon\to 0}\frac{c_{\varepsilon}}{\varepsilon^N}\le\min_{x\in\Lambda}\mathcal{C}_{V(x)}.
$$
\end{proposition}
\begin{proof}
Let $\varphi\in C^{\infty}_c(\R^N)\backslash\{0\},\ \varphi\ge 0$ and define for each $x_0\in\Lambda$
$$
\varphi_{\varepsilon}(x) = \varphi\Big(\frac{x - x_0}{\varepsilon}\Big).
$$
Obviously, supp$\varphi_{\varepsilon}\subset \Lambda$ for small $\varepsilon$ and $\gamma_{\varepsilon}(t) = tT_0\varphi_{\varepsilon}\in \Gamma_{\varepsilon}$ for some $T_0$ large enough. Then we have
\begin{align*}
\frac{c_{\varepsilon}}{\varepsilon^N}\le\frac{\max_{t\in[0,1]}
J_{\varepsilon}(\gamma_{\varepsilon}(t))}{\varepsilon^N}\le \mathcal{L}_{V(x_0)}(t\varphi) + o_{\varepsilon}(1)
\end{align*}
and
$$
\limsup_{\varepsilon\to 0}\frac{c_{\varepsilon}}{\varepsilon^N}\le\inf_{{\varphi\in C^{\infty}_c(\R^N)\backslash\{0\}}\atop{\varphi\ge 0}}\mathcal{L}_{V(x_0)}(t\varphi)=\mathcal{C}_{V(x_0)},
$$
which completes the proof.
\end{proof}

Next, we give the lower bounds of solutions of \eqref{eq2.1}.

%

\begin{proposition}\label{pr3.3}
Let $(u_{\varepsilon_n})$ with $\varepsilon_n>0,\ \varepsilon_n\to 0\ \text{as}\ n\to\infty$ be a family of solutions of $\eqref{eq2.1}$. If for each $k\in \N$, there exists $k$ families of points $\{(x^i_{\varepsilon_n}):1\le i\le k\}$ with $\lim\limits_{n\to\infty}x^i_{\varepsilon_n}=x^i_*$ such that
$$
\liminf_{n\to\infty}\|u_{\varepsilon_n}\|_{L^{\infty}(B_{\varepsilon_n\rho})(x^i_{\varepsilon_n})}>0,\ V(x^i_*)+1>0,\ \ 1\le i\le k,
$$
$$
\liminf_{n\to\infty}\frac{|x^i_{\varepsilon_n} - x^j_{\varepsilon_n}|}{\varepsilon_n} =+\infty,\ \ 1\le i\ne j\le k
$$
and
$$
\limsup_{n\to\infty}\frac{J_{\varepsilon_n}(u_{\varepsilon_n})}{\varepsilon^N_n}<+\infty,
$$
then
$$
\liminf_{n\to\infty}\frac{J_{\varepsilon_n}(u_{\varepsilon_n})}{\varepsilon^N}\ge\sum_{i = 1}^k\mathcal{C}_{V(x^i_*)}.
$$
\end{proposition}

\begin{proof}
Fixing a $1\le i\le k$ and rescaling  the function $u_{\varepsilon_n}$ as $v^i_{n}(x) = u(\varepsilon_n x + x^i_{\varepsilon_n}),\ x\in\R^N$, we have by the estimate in Proposition \ref{pr2.2}  that
$$
\sup_{n}\int_{\R^N}\big(|(-\Delta)^{s/2} v^i_n|^2 + \big(V^i_n(x)+1\big)|v^i_n|^2\big)<+\infty,
$$
where $V^i_n(\cdot)=V(\varepsilon_n \cdot + x^i_{\varepsilon_n})$. Obviously, $v^i_n$ satisfies
\begin{equation}\label{Aeq3.2}
(-\Delta)^s v^i_n + \big(V^i_n(x)+1\big)v^i_n = {G}_1'(\varepsilon_n x + x^i_{\varepsilon_n},v^i_n) -{G}_2'(\varepsilon_n x + x^i_n,v^i_n)\ \ \text{in}\ \R^N.
\end{equation}
 Fixing $R>0$, we have by continuity that
\begin{align*}
&\quad\limsup_{n\to\infty}\|v^i_{n}\|^2_{H^s(B_R)}\\
&\le \limsup_{n\to\infty}(\inf_{\Lambda}(V(x) + 1))^{-1}\int_{\R^N}\big(|(-\Delta)^{s/2} v^i_n|^2 + \big( V^i_n(x)+1\big)|v^i_n|^2\big)<+\infty,
\end{align*}
which says that $(v^i_n)$ is bounded in $H^s_{loc}(\R^N)$ and then by diagonal argument, we can assume without loss of generality that
$
v^i_n\rightharpoonup v^i_*
$
weakly in $H^s_{loc}$ as $n\to \infty$. By
$$
\|v^i_*\|^2_{H^s(B_R)}\le \liminf_{n\to\infty}\|v^i_{n}\|^2_{H^s(B_R)}<+\infty,
$$
we have $v^i_*\in H^s(\R^N)$.

The smoothness of $\Lambda$ implies that the set $\Lambda^i_n=\{x:\varepsilon_n x + x^i_{\varepsilon_n}\in\Lambda\}$ converges to a set $\Lambda^i_*\in\{\emptyset,H,\R^N\}$ as $n\to \infty$, where $H$ is a half plane, by which we have
\begin{align*}
&\quad\int_{\R^N}\Big({G}_1'(\varepsilon x + x^i_{\varepsilon_n},v^i_n)\varphi - {G}_2'(\varepsilon x + x^i_{\varepsilon_n},v^i_n)\varphi\Big)\\
&\to \int_{\R^N}\chi_{\Lambda^i_*}(1 +\log |v^i_*|^2)v^i_*\varphi - \int_{\R^N}\chi_{\R^N\backslash\Lambda^i_*}\max\{2v^i_*,-(1 +\log |v^i_*|^2)v^i_*\}\varphi
\end{align*}
for all $\varphi\in C^{\infty}_c(\R^N)$ as $n\to \infty$.
Then we conclude that $v^i_*$ satisfies the following equation:
\begin{align}\label{eq3.2}
\nonumber&\
\quad(-\Delta)^s v^i_* + (V(x^i_*) + 1)v^i_*\\
&= \chi_{\Lambda^i_*}(x)v^i_*(1+\log|v^i_*|^2)-\chi_{\R^N\backslash\Lambda^i_*}\max\{2v^i_*,-(1 +\log |v^i_*|^2)v^i_*\}\ \ \text{in}\ R^N.
\end{align}
By the similar regularity argument in \cite[Appendix D]{20}, we have
$$
\|v^i_*\|_{L^{\infty}(B_{\rho}(x^i_*))} = \lim_{n\to\infty}|v^i_n\|_{L^{\infty}(B_{\rho}(x^i_n))} = \lim_{n\to\infty}\|u^i_n\|_{L^{\infty}(B_{\varepsilon_n\rho}(x^i_{\varepsilon_n}))}>0,
$$
which implies that $v^i_*$ is nontrivial.

The Euler-Lagrange functional corresponding to \eqref{eq3.2} is
\begin{align*}
J^i_*(u)&=\frac{1}{2}\int_{\R^N}(|(-\Delta)^{s/2}u|^2 + (V(x^i_*) + 1)|u|^2) - \frac{1}{2}\int_{\R^N}\chi_{\Lambda^i_*}|u|^2\log|u|^2\\
 &\quad+ \frac{1}{2}\int_{\R^N\backslash\Lambda^i_*}\mathrm{d}x\int_{0}^{|u(x)|}\max\{4s,-2s(1 + \log s^2)\}ds,
\end{align*}
which implies
\begin{align*}
J^i_*(v^i_*) = \max_{t>0}J^i_*(tv^i_*)\ge \max_{t>0}\mathcal{L}_{V(x^i_*)}(tv^i_*)\ge\mathcal{C}_{V(x^i_*)}.
\end{align*}
Then, after rescaling, we have
\begin{align*}
&\quad \liminf_{n\to\infty}\frac{1}{2\varepsilon^N_n}\int_{B_{\varepsilon_nR}(x^i_{\varepsilon_n})}\Big((
(\varepsilon^{2s}_n|(-\Delta)^{s/2}u_n|^2 + (V(x) + 1)|u_n|^2)\\
&\qquad - \chi_{\Lambda}|u_n|^2\log|u_n|^2+\chi_{\R^N\backslash\Lambda}{G}_2(x,u_n)\Big)\\
&\ge J^i_*(v^i_*) + o_R(1)\ge \mathcal{C}_{V(x^i_*)} + o_R(1).
\end{align*}

Now let us estimate the energy outside $\bigcup_{i=1}^kB_{\varepsilon_n}(x^i_{\varepsilon_n})$.
Choose $\eta$ as another cut-off  function with $\eta_R\equiv0$ in $B_{R}$ and $\eta_R\equiv1$ on $B^c_{2R}$ and define
$$
\eta_{n,R}(\cdot) = \prod_{i=1}^k\eta_R\Big(\frac{\cdot - x^i_{\varepsilon_n}}{\varepsilon_n}\Big).
$$
Testing  \eqref{eq2.1} against with $\eta_{n,R}u_{\varepsilon_n}$, by the definition of penalized function in \eqref{AAAeq1}, 
we find
\begin{align}\label{AAdeq3.5}
\nonumber&\quad\frac{1}{2}\int_{\Big(\bigcup_{i=1}^kB_{\varepsilon_nR}(x^i_{\varepsilon_n})\Big)^c}(\varepsilon^{2s}_n|(-\Delta)^{s/2}u_n|^2 + (V(x) + 1)|u_n|^2)\\
\nonumber&- \frac{1}{2}\int_{\Big(\bigcup_{i=1}^kB_{\varepsilon_nR}(x^i_{\varepsilon_n})\Big)^c\cap\Lambda}
{G}_1(x,u_n)+\frac{1}{2}\int_{\Big(\bigcup_{i=1}^kB_{\varepsilon_nR}(x^i_{\varepsilon_n})\Big)^c\cap\R^N\backslash\Lambda}
{G}_2(x,u_n)\\
\nonumber& := T^2_{n,R}+\frac{1}{2}\int_{{\bigcup_{i=1}^k\Big(B_{2\varepsilon_nR}(x^i_{\varepsilon_n})\backslash B_{\varepsilon_nR}(x^i_{\varepsilon_n})\Big)}}(1-\eta_{n,R})
(\varepsilon^{2s}_n|(-\Delta)^{s/2}u_n|^2 + (V(x) + 1)|u_n|^2)\\
\nonumber&\quad + \frac{1}{2}\int_{\Lambda}\eta_{n,R}|u_n|^2(1+\log|u_n|^2) - \frac{1}{2}{\int_{\R^N}}\chi_{\Big(\bigcup_{i = 1}^kB_{\varepsilon_nR}(x^i_{\varepsilon_n})\Big)^c\cap\Lambda}|u_n|^2\log|u_n|^2\\
\nonumber&\quad-\frac{1}{2}\int_{\R^N\backslash\Lambda}\eta_{n,R}(x)\max\{2|u_n|^2,-|u_n|^2(1+\log|u_n|^2)\}\\
&\quad + \frac{1}{2}\int_{\R^N\backslash\Lambda}\chi_{\Big(\bigcup_{i = 1}^kB_{\varepsilon_nR}(x^i_{\varepsilon_n})\Big)^c}\mathrm{d}x\int_{0}^{|u_n|}\max\{4t_+,-2t_+(1+\log |t_+|^2)\}\\
\nonumber&\ge T^2_{n,R}
 + \frac{1}{2}\int_{\Lambda}\eta_{n,R}|u_n|^2(1+\log|u_n|^2) - \frac{1}{2}{\int_{\Lambda}}\chi_{\Big(\bigcup_{i = 1}^kB_{\varepsilon_nR}(x^i_{\varepsilon_n})\Big)^c}|u_n|^2\log|u_n|^2\\
\nonumber&\quad-\frac{1}{2}\int_{\R^N\backslash\Lambda}\eta_{n,R}(x)\max\{2|u_n|^2,-|u_n|^2(1+\log|u_n|^2)\}\\
\nonumber&\quad + \frac{1}{2}\int_{\R^N\backslash\Lambda}\chi_{\Big(\bigcup_{i = 1}^kB_{\varepsilon_nR}(x^i_{\varepsilon_n})\Big)^c}\max\{2|u_n|^2,-|u_n|^2(1+\log |u_n|^2)\}\mathrm{d}x\\
\nonumber& \ge T^2_{n,R} 
+ \frac{1}{2}\int_{\bigcup_{i=1}^k\Big(B_{2\varepsilon_nR}(x^i_{\varepsilon_n})\backslash B_{\varepsilon_nR}(x^i_{\varepsilon_n})\Big)\cap\Lambda}(\eta_{n,R}-1)|u_n|^2\log|u_n|^2\\
\nonumber&\quad +\frac{1}{2}\int_{{\bigcup_{i=1}^k\Big(B_{2\varepsilon_nR}(x^i_{\varepsilon_n})\backslash B_{\varepsilon_nR}(x^i_{\varepsilon_n})\Big)}\cap\R^N\backslash\Lambda}(\eta_{n,R}-1)\max\{2|u_n|^2,-|u_n|^2(1 + \log|u_n|^2)\},
\end{align}
where
\begin{align*}
&T^2_{n,R}= \frac{\varepsilon^{2s}_n}{2}\int_{\R^N}\mathrm{d}x\int_{\R^N}\frac{(\eta_{n,R}(y)-\eta_{n,R}(x))u_n(y)(u_n(x)-u_n(y))}{|x-y|^{N+2s}}\mathrm{d}y.
\end{align*}
In Appendix A, we will prove by fractional Hardy inequality \eqref{Adeq1.4} that
\begin{equation}\label{Adeq3.5}
\limsup_{n\to\infty}\frac{T^2_{n,R}}{\varepsilon^N_n}\ge o_R(1).
\end{equation}
Hence, by the fact that $v^i_n\to v^i_*$ strongly in $L^q_{loc}(\R^N)$ with $1<q<2^*_s$, we conclude that
\begin{align*}
&\lim_{n\to\infty}\Big(\quad\frac{1}{2\varepsilon^N_n}\int_{\Big(\bigcup_{i=1}^kB_{\varepsilon_nR}(x^i_{\varepsilon_n})\Big)^c}(\varepsilon^{2s}_n|(-\Delta)^{s/2}u_n|^2 + (V(x) + 1)|u_n|^2)\\
&- \frac{1}{2\varepsilon^N_n}\int_{\Big(\bigcup_{i=1}^kB_{\varepsilon_nR}(x^i_{\varepsilon_n})\Big)^c}
{G}_1(x,u_n)+\frac{1}{2\varepsilon^N_n}\int_{\Big(\bigcup_{i=1}^kB_{\varepsilon_nR}(x^i_{\varepsilon_n})\Big)^c}
{G}_2(x,u_n)\Big)\\
&\ge o_R(1) - C\int_{B_{2R}\backslash B_{R}}(|v^i_*| + |v^i_*|^{q})\\
&=o_R(1).
\end{align*}

Finally, by the analysis above, we have
$$
\liminf_{n\to\infty}\frac{J_{\varepsilon_n}(u_{\varepsilon_n})}{\varepsilon^N_n}\ge \sum_{i = 1}^k\mathcal{C}_{V(x^i_*)} + o_R(1),
$$
the conclusion then follows by letting $R\to\infty$.

\end{proof}

\begin{remark}

It is easy to check that
$$
\int\max\{f,g\}\ge \max\{\int f,\int g\},
$$
which and the skillful choice of the truncated function in \eqref{AAAeq1} play a key role in the proof of \eqref{AAdeq3.5}. The estimates \eqref{Adeq3.5} is trivial in the classical case, but delicate under the nonlocal effect of $(-\Delta)^s(0<s<1)$, some skillful global estimates will be involved(e.g., the $L^2$ information of $u_{\varepsilon_n}$ outside $\Lambda$, see \eqref{Adeq3.3} in Appendix for example).
\end{remark}

Now we prove the concentration of $u_{\varepsilon}$.
\begin{lemma}\label{le3.4}
Let $\rho>0$ and $u_{\varepsilon}$ be the penalized solution given by Lemma \ref{le2.3}. There exists a family of points $(x_{\varepsilon})\subset\Lambda$ such that

$(i)$ $\liminf\limits_{\varepsilon\to 0}\|u_{\varepsilon}\|_{L^{\infty}(B_{\varepsilon\rho}(x_{\varepsilon}))}>0$.

$(ii)$ $\lim\limits_{\varepsilon\to 0}V(x_{\varepsilon}) = \inf_{\Lambda} V$.

$(iii)$ $\lim\limits_{{R\to\infty}\atop{\varepsilon\to0}}\|u_{\varepsilon}\|_{L^{\infty}(U\backslash B_{\varepsilon R}(x_{\varepsilon}))}=0$.

\end{lemma}

\begin{proof}
Easily, we have
\begin{align*}
0&<\int_{\R^N}(\varepsilon^{2s}|(-\Delta)^{s/2}u_{\varepsilon}|^2 + (V(x) + 1)|u_{\varepsilon}|^2)\le \int_{\Lambda\cap\{x:u_{\varepsilon}(x)> e^{-1/2}\}}|u_{\varepsilon}|^2(1 + \log|u_{\varepsilon}|^2)\\
&\qquad\le \|1 + \log|u_{\varepsilon}|^2\|_{L^{\infty}(\Lambda\cap\{x:u_{\varepsilon}(x)> e^{-1/2}\})}\int_{\Lambda}|u_{\varepsilon}|^2,
\end{align*}
which and the similar regularity assertion in \cite{20} imply that there exists $x_{\varepsilon}\in \overline{\Lambda}$ such that
$$
u_{\varepsilon}(x_{\varepsilon}) = \sup_{x\in\Lambda}u_{\varepsilon}(x)\ \ \text{and}\ \liminf_{\varepsilon\to 0}\|u_{\varepsilon}\|_{L^{\infty}(B_{\varepsilon\rho}(x_{\varepsilon}))}>0.
$$
This proves $(i)$.

For $(ii)$, assuming without loss of generality that $\lim\limits_{\varepsilon\to 0}x_{\varepsilon} = x_*$, by the lower and upper bounds of $u_{\varepsilon}$ in Propositions \ref{pr3.2} and \ref{pr3.3}, we have
$$
\min_{x\in \Lambda}\mathcal{C}_{V(x)}\ge \liminf_{\varepsilon\to 0}\frac{J_{\varepsilon}(u_{\varepsilon})}{\varepsilon^N}\ge\mathcal{C}_{V(x_*)},
$$
which implies $V(x_*) = \min_{x\in\Lambda}V(x)$.

For $(iii)$, if it is not true, then one will get a contradiction like
$$
\min_{x\in \Lambda}\mathcal{C}_{V(x)}\ge \liminf_{\varepsilon\to 0}\frac{J_{\varepsilon}(u_{\varepsilon})}{\varepsilon^N}\ge\mathcal{C}_{V(x_*)} + \mathcal{C}_{V(y_*)}
$$
for some $y_*\in \overline{U}$ by Proposition \ref{pr3.3}.
\end{proof}

\subsection{Back to the origin problem}
In this subsection, we use Lemma \ref{le3.4} to linearize equation \eqref{eq2.1} and then use the well-known decay estimates of positive solutions of fractional Schr\"odinger equations  to show that \eqref{Adeq3.1} is true.

Noting that by the regular assertion in Appendix D of \cite{20}, we can assume that
\begin{equation}\label{eq5.1}
\sup_{\Lambda}u_{\varepsilon}(x)\le C<\infty,
\end{equation}
where $C$ is a positive constant. Hence by Lemma \ref{le3.4},
we can linearize the penalized equation \eqref{eq2.1} as follows.
\begin{proposition}\label{pr5.1}
Let $\varepsilon>0$ be small enough, $x_{\varepsilon}$ be the point given by Lemma \ref{le3.4}. Then there exists $R>0$ such that
$$
\left\{
  \begin{array}{ll}
   \varepsilon^{2s}(-\Delta)^su_{\varepsilon} + \min\{\lambda,2\} u_{\varepsilon}\le 0, & \text{in}\ \R^N\backslash B_{\varepsilon R}(x_{\varepsilon}) \\
    u_{\varepsilon}\le C, &\text{in}\ B_{\varepsilon R}(x_{\varepsilon})
  \end{array}
\right.
$$
\end{proposition}

\begin{proof}
For $\varepsilon>0$ small enough, by Lemma \ref{le3.4}, there exists $R>0$ such that $
(1 + \log|u_{\varepsilon}|^2)\le 0
$
for all $x\in U\backslash B_{\varepsilon R}(x_{\varepsilon})$,
the conclusion then follows by the penalized function in \eqref{AAAeq1} and inserting \eqref{eq5.1} into \eqref{eq2.1}.
\end{proof}

At last, we prove Theorem \ref{th1.1}.

\textbf{Completes the proof of Theorem \ref{th1.1}}.

By Proposition \ref{pr5.1}, the rescaling function $v_{\varepsilon} = u_{\varepsilon}(\varepsilon x + x_{\varepsilon}$) satisfies
\begin{equation*}
\left\{
  \begin{array}{ll}
   (-\Delta)^s v_{\varepsilon} + \lambda v_{\varepsilon}\le 0, & \text{in}\ \R^N\backslash B_{R} \\
    v_{\varepsilon}\le C, &\text{in}\ B_{R}.
  \end{array}
\right.
\end{equation*}
Then by the well-known decay estimates of fractional Schr\"odinger equation(see \cite[Appendix D]{20} for example), we have
\begin{equation}\label{Adeq3.7}
v_{\varepsilon}(x)\le \frac{C}{1 + |x|^{N+2s}},\ x\in\R^N.
\end{equation}
Consequently, it holds
$$
u_{\varepsilon}(x)\le \frac{C\varepsilon^{N+2s}}{\varepsilon^{N+2s} + |x - x_{\varepsilon}|^{N+2s}},
$$
from which, we can conclude that
$$
\max\{2u_{\varepsilon},-\big(1 + \log(u_{\varepsilon})^2\big)\} = -\big(1 + \log(u_{\varepsilon})^2\big)\ \ \text{for all}\ x\in\R^N\backslash\Lambda
$$
if $\varepsilon>0$ is small. This proves that $u_{\varepsilon}$ solves the origin problem \eqref{eq1.1} and complete the proof of this paper.

\vspace{0.5cm}
\appendix
  \renewcommand{\appendixname}{Appendix~\Alph{section}}
  \section{}
In this section, we are going to verify \eqref{Adeq2.5} and \eqref{Adeq3.5}. The fractional Hardy inequality \eqref{Adeq1.4} will be involved in the proof. We first give the proof of \eqref{Adeq2.5}. By Cauchy inequality, we have
\begin{align*}
T_1(\eta)
&\le 2{\varepsilon^{2s}}\int_{\R^N}\mathrm{d}x\int_{\R^N} \frac{|u_n(x)|^2(\eta(x)-\eta(y))^2}{|x-y|^{N+2s}}\mathrm{d}y + C\|u_n\|^2_{V,\varepsilon}\\
&:= T_{11}(\eta) + C\|u_n\|^2_{V,\varepsilon}.
\end{align*}

By decomposition, denoting $f(x,y)=\frac{|u_n(x)|^2(\eta(x)-\eta(y))^2}{|x-y|^{N+2s}}$, we have
\begin{align*}
\frac{1}{2}T_{11}(\eta)
&=\varepsilon^{2s}\int_{\Lambda^{\delta}}\mathrm{d}x\int_{\R^N}f(x,y)\mathrm{d}y +  \varepsilon^{2s}\int_{\R^N\backslash\Lambda^{\delta}}\mathrm{d}x\int_{\Lambda^{\delta}}f(x,y)\mathrm{d}y\\
&:=\sum_{i=1}^2T_{11i}(\eta).
\end{align*}

Since
\begin{align*}
\int_{\R^N}\frac{(\eta(x) - \eta(y))^2}{|x-y|^{N+2s}}\mathrm{d}y
&=\int_{B_1(x)}\frac{(\eta(x) - \eta(y))^2}{|x-y|^{N+2s}}\mathrm{d}y+\int_{\R^N\backslash B_1(x)}\frac{(\eta(x) - \eta(y))^2}{|x-y|^{N+2s}}\mathrm{d}y\\
&\le C\Big(\int_{B_1(0)}\frac{1}{|z|^{N+2s-2}}\mathrm{d}y+\int_{\R^N\backslash B_1(0)}\frac{1}{|z|^{N+2s}}\mathrm{d}y\Big)\\
&\le C,
\end{align*}
we have
\begin{eqnarray}\label{eqA.1}
\begin{split}
T_{111}(\eta)
\le C\varepsilon^{2s}\int_{\Lambda^{\delta}}|u_n(x)|^2
\le C\|u_n\|^2_{V,\varepsilon}.
\end{split}
\end{eqnarray}

For $T_{112}(\eta)$, similar to the estimates above, by fractional Hardy inequality \eqref{Adeq1.4}, we have
\begin{eqnarray}\label{eqA.2}
\begin{split}
T_{112}(\eta)
&\le  C\varepsilon^{2s}\int_{\R^N\backslash\Lambda^{2\delta}}|u_n(x)|^2\mathrm{d}x\int_{\Lambda^{\delta}}\frac{1}{|x-y|^{N+2s}}\mathrm{d}y\\
&\le  C\varepsilon^{2s}\int_{\R^N\backslash\Lambda^{2\delta}}\frac{|u_n(x)|^2}{|x|^{2s}}\mathrm{d}x\int_{\Lambda^{\delta}}\frac{|x|^{2s}}{|x-y|^{N+2s}}\mathrm{d}y\\
&\le C\varepsilon^{2s}\int_{\R^N\backslash\Lambda^{2\delta}}\frac{|u_n(x)|^2}{|x|^{2s}}\mathrm{d}x\int_{\Lambda^{\delta}}\frac{|x|^{2s}}{\big(dist(x,\partial\Lambda^{2\delta}) + \delta + dist(y,\partial\Lambda^{\delta})\big)^{N+2s}}\mathrm{d}y\\
&\le C\|u_n\|^2_{V,\varepsilon}.
\end{split}
\end{eqnarray}

By \eqref{eqA.1} and \eqref{eqA.2}, we conclude that`
$$
T_1(\eta)\le C\|u_n\|^2_{V,\varepsilon}.
$$

Next, we give the proof of \eqref{Adeq3.5}. A change of variable tells us
\begin{eqnarray*}
\begin{split}
T^2_{n,R}& = \int_{\mathbb{R}^N}u_n(y)\,{\mathrm{d}}y\int_{\mathbb{R}^N}\frac{(u_n(x) - u_n(y))(\psi_{n,R}(x) - \psi_{n,R}(y))}{|x - y|^{N + 2s}}\,{\mathrm{d}}x\\
 & = \varepsilon^{N - 2s}_n\sum_{l = 1}^k\int_{\mathbb{R}^N}v^l_n(y)\beta^l_n(y)\,{\mathrm{d}}y\int_{\mathbb{R}^N}\frac{\alpha^l_n(x)(v^l_n(x) - v^l_n(y))(\eta_R(x) - \eta_R(y))}{|x - y|^{N + 2s}}\,{\mathrm{d}}x,
\end{split}
\end{eqnarray*}
where the functions $\beta^l_n$ and $\alpha^l_n$ are defined skillfully as
$$
\beta^l_n(y) = \prod^{l - 1}_{s = 0}\eta\Big(\frac{y}{R} + \frac{x^l_n - x^s_n}{\varepsilon_n R}\Big),\,\,\,\beta^0_n(y)\equiv 1
$$
and
$$
\alpha^l_n(x) = \prod^{k}_{s = l + 1}\eta\Big(\frac{x}{R} + \frac{x^l_n - x^s_n}{\varepsilon_n R}\Big),\,\,\, \alpha^k_n(x) \equiv 1.
$$
Following, we have
\begin{align*}
\varepsilon^{2s - N}_nT^2_{n,R}& = \sum_{l = 1}^k\int_{B_{2R}} v^l_n(y) \beta^l_n(y) \,{\mathrm{d}}y \int_{B^c_{2R}}\frac{\alpha^l_n(x)(v^l_n(x) - v^l_n(y))(\eta_R(x) - \eta_R(y))}{|x - y|^{N + 2s}}\,{\mathrm{d}}x
\\[1mm]
   &\quad + \sum_{l = 1}^k \int_{B^c_{2R}}v^l_n(y)\beta^l_n(y)\,{\mathrm{d}}y  \int_{B_{2R}}\frac{\alpha^l_n(x)(v^l_n(x) - v^l_n(y))(\eta_R(x) - \eta_R(y))}{|x - y|^{N + 2s}}\,{\mathrm{d}}x
   \\[1mm]
   &\quad + \sum_{l = 1}^k\int_{B_{2R}}v^l_n(y)\beta^l_n(y)\,{\mathrm{d}}y   \int_{B_{2R}}\frac{\alpha^l_n(x)(v^l_n(x) - v^l_n(y))(\eta_R(x) - \eta_R(y))}{|x - y|^{N + 2s}}\,{\mathrm{d}}x\\[1mm]
   &:=T^{21}_{n,R} +T^{22}_{n,R} + T^{23}_{n,R}.
\end{align*}
By the choice of $\eta_R$ and  $\lim\limits_{n\to\infty}\frac{|x^l_n - x^s_n|}{\varepsilon_n} = \infty$ if $l\neq s$, for $n$ large, we have
\begin{align*}
T^{21}_{n,R}&= \sum_{l = 1}^k\int_{B_{2R}\backslash B_R}v^l_n(y)\beta^l_n(y)\,{\mathrm{d}}y\int_{B^c_{2R}}\frac{\alpha^l_n(x)(v^l_n(x) - v^l_n(y))(1 - \eta_R(y))}{|x - y|^{N + 2s}}\,{\mathrm{d}}x\\[1mm]
       &\quad + \sum_{l = 1}^k\int_{{B_R}}v^l_n(y)\beta^l_n(y)\,{\mathrm{d}}y\int_{B^c_{2R}}\frac{\alpha^l_n(x)(v^l_n(x) - v^l_n(y))}{|x - y|^{N + 2s}}\,{\mathrm{d}}x\\[1mm]
       &= \sum_{l = 1}^k\int_{B_{2R}\backslash B_R}v^l_n(y)\,{\mathrm{d}}y\int_{B^c_{2R}}\frac{\alpha^l_n(x)(v^l_n(x) - v^l_n(y))(1 - \eta_R(y))}{|x - y|^{N + 2s}}\,{\mathrm{d}}x\\[1mm]
       &\quad + \sum_{l = 1}^k\int_{{B_R}}v^l_n(y)\,{\mathrm{d}}y\int_{B^c_{2R}}\frac{\alpha^l_n(x)(v^l_n(x) - v^l_n(y))}{|x - y|^{N + 2s}}\,{\mathrm{d}}x\\[1mm]
       :&=T^{211}_{n,R} +T^{212}_{n,R}
\end{align*}
and
\begin{align*}
        T^{22}_{n,R}
          =& \sum_{l = 1}^k\int_{B^c_{2R}}v^l_n(y)\beta^l_n(y)\,{\mathrm{d}}y\int_{B_{2R}\backslash B_R}\frac{\alpha^l_n(x)(v^l_n(x) - v^l_n(y))(\eta_R(x) - 1)}{|x - y|^{N + 2s}}\,{\mathrm{d}}x\\[1mm]
         & - \sum_{l = 1}^k\int_{B^c_{2R}}v^l_n(y)\beta^l_n(y)\,{\mathrm{d}}y\int_{B_R}\frac{\alpha^l_n(x)(v^l_n(x) - v^l_n(y))}{|x - y|^{N + 2s}}\,{\mathrm{d}}x \\[1mm]
         =& \sum_{l = 1}^k\int_{B^c_{2R}}v^l_n(y)\beta^l_n(y)\,{\mathrm{d}}y\int_{B_{2R}\backslash B_R}\frac{(v^l_n(x) - v^l_n(y))(\eta_R(x) - 1)}{|x - y|^{N + 2s}}\,{\mathrm{d}}x\\[1mm]
         & - \sum_{l = 1}^k\int_{B^c_{2R}}v^l_n(y)\beta^l_n(y)\,{\mathrm{d}}y\int_{B_R}\frac{v^l_n(x) - v^l_n(y)}{|x - y|^{N + 2s}}\,{\mathrm{d}}x \\[1mm]
                  :=&T^{221}_{n,R} + T^{222}_{n,R}.
\end{align*}
Also, for large $n$,
\begin{align*}
T^{23}_{n,R} =&\sum_{l = 1}^k\int_{B_{2R}\backslash{B_R}}v^l_n(y)\,{\mathrm{d}}y\int_{B_{2R}}\frac{(v^l_n(x) - v^l_n(y))(\eta_R(x) - \eta_R(y))}{|x - y|^{N + 2s}}\,{\mathrm{d}}x\\[1mm]
& + \sum_{l = 1}^k\int_{B_{2R}}v^l_n(y)\,{\mathrm{d}}y\int_{B_{2R}\backslash{B_R}}\frac{(v^l_n(x) - v^l_n(y))(\eta_R(x) - \eta_R(y))}{|x - y|^{N + 2s}}\,{\mathrm{d}}x\\[1mm]
&+ \sum_{l = 1}^k\int_{B_{2R}\backslash{B_R}}v^l_n(y)\,{\mathrm{d}}y\int_{B_{2R}\backslash B_{R}}\frac{(v^l_n(x) - v^l_n(y))(\eta_R(x) - \eta_R(y))}{|x - y|^{N + 2s}}\,{\mathrm{d}}x\\[1mm]
:=& T^{231}_{n,R} + T^{232}_{n,R} + T^{233}_{n,R}.
\end{align*}

For $|T^{2i2}_{n,R}|,\ i = 1,2$, it holds
\begin{align*}
\nonumber    &\limsup_{n\to\infty}  |T^{2i2}_{n,R}|  \leq CR^{-2s} + \limsup_{n\to\infty}2\sum_{l = 1}^k\int_{B^c_{2R}}\,{\mathrm{d}}y\int_{B_{R}}\frac{ (v^l_n(y))^2}{|x-y|^{N + 2s}}\,{\mathrm{d}}x.
\end{align*}
By the fractional Hardy inequality \eqref{Adeq1.4} and letting $\widetilde{R} = R^{\frac{N + 1}{N}}$, we find
\begin{align}\label{Adeq3.3}
\nonumber&\quad \limsup_{n\to\infty}\int_{B^c_{2R}}\big(v^l_n(y))^2\,{\mathrm{d}}y\int_{B_R}\frac{1}{|x - y|^{N + 2s}}\,{\mathrm{d}}x\\[1mm]
\nonumber&\leq C\limsup_{n\to\infty}\int_{B^c_{2R}}\big(v^l_n(y))^2\frac{R^N}{|y|^{N + 2s}}\,{\mathrm{d}}y\\[1mm]
\nonumber& \le C\limsup_{n\to\infty}\int_{B_{\widetilde{R}}\backslash B_{2R}}\big(v^l_n(y))^2\frac{R^N}{|y|^{N + 2s}}\,{\mathrm{d}}y + C\limsup_{n\to\infty} \int_{B^c_{\widetilde{R}}}\big(v^l
_n(y))^2\frac{R^N}{|y|^{N + 2s}}\,{\mathrm{d}}y\\[1mm]
&\leq C\limsup_{n\to\infty}\int_{B_{\widetilde{R}}\backslash B_{2R}}\big(v^l_n(y))^2\,{\mathrm{d}}y + C \limsup_{n\to\infty}\int_{B^c_{\widetilde{R}}}\frac{(v^l_n(y))^2}{|y|^{2s}}\frac{R^N}{|y|^{N}}\,{\mathrm{d}}y\\[1mm]
\nonumber&\leq  C\int_{B_{\widetilde{R}}\backslash B_{2R}}\big(v^l_*(y))^2\,{\mathrm{d}}y + \frac{C}{R}\\[1mm]
\nonumber& = o_R(1).
\end{align}

Noting that for each $\Omega\subset\subset\R^N$ is smooth, it holds
$$
v^l_n\rightharpoonup v^l_*\ \text{weakly in}\ W^{s,2}(\Omega).
$$
Then, for $T^{211}_{n,R}$, by the estimates of  $T^{2i2}_{n,R}$($i=1,2$), we have

\begin{align*}
&\limsup_{n\to\infty}  |T^{211}_{n,R}| \nonumber
\\[1mm]
&\leq
  \limsup_{n\to\infty}  \sum_{l = 1}^k\int_{B_{2R}\backslash B_R}\,{\mathrm{d}}y\int_{B^c_{2R}}\frac{|v^l_n(x) - v^l_n(y)|^2}{|x - y|^{N + 2s}}\,{\mathrm{d}}x \nonumber
  \\[1mm]
 &
 + \limsup_{n\to\infty}\sum_{l = 1}^k \int_{B_{2R}\backslash B_R}(v^l_n(y))^2\,{\mathrm{d}}y \int_{B^c_{2R}}\frac{(1 - \eta_{R}(y))^2}{|x - y|^{N + 2s}}\,{\mathrm{d}}x \nonumber
  \\[1mm]
 & \leq  \limsup_{n\to\infty}  \sum_{l = 1}^k  \int_{B_{2R}\backslash B_R}\,{\mathrm{d}}y\int_{B^c_{2R}\cap B_{4R}}\frac{|v^l_n(x) - v^l_n(y)|^2}{|x - y|^{N + 2s}}\,{\mathrm{d}}x\nonumber
  \\[1mm]
 &+\limsup_{n\to\infty}\sum_{l = 1}^k \int_{B_{2R}\backslash B_R}\,{\mathrm{d}}y\int_{B^c_{4R}}\frac{|v^l_n(x) - v^l_n(y)|^2}{|x - y|^{N + 2s}}\,{\mathrm{d}}x + C\limsup_{n\to\infty}\sum_{l = 1}^k\int_{B_{2R}\backslash B_R}(v^l_n(y))^2\,{\mathrm{d}}y\nonumber
 \\[1mm]
 & \leq  \limsup_{n\to\infty}\sum_{l = 1}^k\int_{B_{4R}\backslash B_R}\,{\mathrm{d}}y\int_{B_{4R}\backslash B_R}\frac{|v^l_n(x) - v^l_n(y)|^2}{|x - y|^{N + 2s}}\,{\mathrm{d}}x + C\sum_{l = 1}^k\int_{\R^N}(v^*_l(y))^2\mathrm{d}y\nonumber
 \\[1mm]
& \leq C\int_{B_{4R}\backslash B_R}\,{\mathrm{d}}y\int_{R^N}\frac{|v^l_*(x) - v^l_*(y)|^2}{|x - y|^{N + 2s}}\,{\mathrm{d}}x + o_R(1)\\
&=o_R(1).
\end{align*}
Similarly, we get
\begin{align*}
    \limsup_{n\to\infty}  |T^{221}_{n,R}| &\leq o_R(1)
\end{align*}
and
\begin{align*}
  \limsup_{n\to\infty}|T^{231}_{n,R} + T^{232}_{n,R} + T^{233}_{n,R}|\leq o_R(1).
\end{align*}
Therefore
\begin{equation*}
\varepsilon^{2s - N}T^{2}_{n,R} \leq o_R(1).
\end{equation*}

\vspace{0.5cm}



\vspace{1cm}



\begin{thebibliography}{99}
\bibitem{100}
Alves, C., Miyagaki, O.: Existence and concentration of solution for a class of fractional elliptic equation in $\mathbb{R}^N$ via penalization method. Calc. Var. Partial Differential Equations \textbf{55}, 1-19 (2016)

\bibitem{APX}
{An, X., Peng, S., Xie, C}: {Semi-classical solutions for fractional Schr\"odinger equations with potential vanishing at infinity}. J. Math. Phys. \textbf{60}, 021501 (2019)

\bibitem{A.H.Ardila-NA-2017}
A. H. Ardila, Existence and stability of standing waves for nonlinear fractional
Schr\"odinger equation with logarithmic nonlinearity. Nonlinear Analysis \textbf{155} (2017) 52-64



%
%
%
%
%
%
%
%
%
%





\bibitem{A.Cotsiilis-N.Tavoularis-2005}
A. Cotsiolis, N. Tavoularis, On logarithmic Sobolev inequalities for higher order fractional derivatives. Comptes Rendus de l'Academie des Sciences Paris
Serie I 2005; \textbf{340}: 205-208.

\bibitem{P.D'Avenia-CCM-2014}
P. D'Avenia,  E. Montefusco, M. Squassina, On the logarithmic Schr\"odinger equation. Commun. Contemp. Math. \textbf{16}(2), 1350032 (2014)





\bibitem{P.D'Avenia-MMAS-2015}
P. D'Avenia, M. Squassina, M. Zenari, Fractional logarithmic Schr\"odinger equations. Math. Meth. Appl. Sci.  \textbf{2015}, 38 5207-5216

\bibitem{DZ-2000}
M. Degiovanni, S. Zani, Multiple solutions of semilinear elliptic equations with one-sided growth conditions, nonlinear operator theory. Math. Comput. Model. \textbf{32}, 1377-1393 (2000)









\bibitem{4}
E. Di Nezza,  G. Palatucci, E. Valdinoci, {Hitchhikers guide to the fractional Sobolev spaces}. Bull. Sci. Math. \textbf{136} 521-573  (2012)

\bibitem{FS}
L. Frank, R. Seiringer,  Non-linear ground state representations
and sharp Hardy inequalities. J. Funct. Anal. \textbf{255} 3407-3430(2008)

\bibitem{20}
L. Frank, E. Lenzmann,  L. Silvestre,  {Uniqueness of radial solutions for the fractional Laplacians}. Comm. Pure. Appl. Math. \textbf{69}  1671-1726(2016)







\bibitem{2}
N. Laskin, {Fractional Schr\"{o}dinger equation}. Phys. Lett. A \textbf{268} 298-305 (2000)

\bibitem{3}
N. Laskin, {Fractional quantum mechanics and Levy path integrals}. Phys. Lett. A \textbf{268} 298-305 (2000)


\bibitem{J.Serrin-M.Tang-IUMJ-2000}
J. Serrin, M. Tang, Uniqueness of ground states for quasilinear elliptic equations.
Indiana Univ. Math. J. \textbf{49}(3), 897-923 (2000)

\bibitem{W.Shuai-JMP-2021}
W. Shuai, Existence and multiplicity of solutions for
logarithmic Schr\"odinger equations with potential. J. Math. Phys. \textbf{62}, 051501 (2021)


\bibitem{M.Squassina-A.Szulkin-2015}
M. Squassina, A. Szulkin, Multiple solutions to logarithmic Schr\"odinger equations with periodic potential. Calc. Var. Partial Differential Equations \textbf{54} (2015), 585-597.

\bibitem{A.Szulkin-1986}
A. Szulkin, Minimax principles for lower semicontinuous functions and applications to nonlinear boundary value problems. Ann. Inst. H. Poincar$\acute{e}$ Anal. Non Lin$\acute{e}$aire \textbf{3}, 77-109 (1986)

\bibitem{W.C.Troy-ARMA-2016}
W. C. Troy, Uniqueness of Positive Ground State Solutions of the Logarithmic Schr\"odinger Equation. Arch. Rational Mech. Anal. \textbf{222} (2016) 1581-1600





\bibitem{Z.Q.-Wang-C.Zhang-ARMA-2019}
Z.-Q Wang, C. X. Zhang, Convergence from power-law to logarithm-law in nonlinear scalar field equations. Arch. Ration. Mech. Anal. \textbf{231}(2019), no. 1, 45-61.



\bibitem{MW}
M. Willem, Minimax Theorems. Progr. Nonlinear Differential Equations Appl., vol. 24, Birkh$\ddot{a}$user, Boston, MA,
1996.

\bibitem{K.G.Zlo-2010}
K. G. Zloshchastiev, Logarithmic nonlinearity in theories of quantum gravity: origin of time and observational consequences. Gravit. Cosmol. \textbf{16}, 288-297 (2010)

\bibitem{C.Zhang-X.Zhang-CV-2020}
C. Zhang, X. Zhang, Bound states for logarithmic Schr\"odinger equations with potentials unbounded below. Calc. Var. (2020) 59:23













\end{thebibliography}
\end{document}